\documentclass[10pt]{amsart}
\usepackage{amscd, amsfonts, amsmath, amsthm, amssymb}
\usepackage{mathtools}
\usepackage{tabularx}
\usepackage{longtable}
\usepackage{enumerate}
\usepackage{booktabs}
\usepackage{comment}
\usepackage[table]{xcolor}
\usepackage{cases}

\usepackage
[  backref,
linkcolor = red,
letterpaper = true,
hyperindex = true,
hidelinks = true,
linktocpage = true
]
{hyperref}

\newcommand{\Z}{\mathbb{Z}}
\newcommand{\F}{\mathbb{F}}

\newcommand{\R}{\mathbb{R}}

\theoremstyle{plain}
\newtheorem{theorem}{Theorem}[section]
\newtheorem{conjecture}[theorem]{Conjecture}

\theoremstyle{definition}

\newtheorem{remark}[theorem]{Remark}
\newtheorem{remarks}[theorem]{Remarks}

\numberwithin{table}{theorem}

\begin{document}

\author{Jos\'e Alejandro Lara Rodr\'iguez}

\title[Two term relations between multizeta of depth two for $\mathbb{F}_q{[t]}$]{Two term relations between multizeta of depth two for $\F_q[t]$}

\address{
Facultad de Matem\'aticas, Universidad Aut\'onoma de Yucat\'an, Perif\'erico
Norte, Tab. 13615,
M\'erida, Yucat\'an, M\'exico, alex.lara@correo.uady.mx}


\date{\today}




\begin{abstract}

We focus on multizeta values of depth two for $\F_q[t]$, where the ratio with another multizeta value of depth two is rational.  In characteristic 2, we prove some extra relations between multizeta values of depth 2 and the same weight.

\end{abstract}

\maketitle

\section{Introduction}

In 1775, more than 30 years after Euler introduced the zeta values $\zeta(s)$, he introduced and studied the multizeta values
\begin{align*}
\zeta(s_1, \dotsc, s_r) & := \zeta_{\Z}(s_1, \dotsc, s_r)\\& :=
\sum _{n_1 > \dotsb > n_r>0} \frac{1}{n_1^{s_1} \dotsm n_r^{s_r} } \in \R, \quad (s_i\in \Z, \, s_i \ge 1, s_1>1).
\end{align*}
Here  $r$ is called the depth  and $\sum s_i$ the weight. 
We say that a multizeta value $\zeta(s_1, \dotsc, s_r)$ is {\it zeta-like} if the ratio $\zeta(s_1, \dotsc, s_r)/\zeta(s_1+ \dotsb + s_r)$ is rational. A multizeta value of weight $w$ is called {\it Eulerian} if it is a rational multiple of $(2\pi i)^{w}$.    
Euler  discovered interesting relations such as  $\zeta(3,1) = \zeta(4)/4$ and $\zeta(2,1) = \zeta(3)$. These relations   generalizes to $\zeta(2, \{1\}_k) = \zeta(k+2)$ (special case of Hoffman-Zagier duality relation) and $\zeta(\{3,1\}_k) = \zeta(\{2\}_{2k}) /(2k+1)$ (Broadhurst's result, conjectured by Zagier) which is known to be Eulerian as $\zeta(\{2\}_k) = \pi ^{2k}/(2k+1)!$. Here we use the standard short-form ${X}_k$ for the tuple $X$ repeated $k$ times (See \cite[Remark after Conjecture 4.3]{LT14}.) 

The relations among multizeta values have been studied extensively for the past three decades. They have resurfaced with renewed interest because of their connections with various fundamental objects and structures in mathematics and mathematical physisc.   In some sense, the relations have been at least conjecturally understood, though much remains to be proved. See e.g., \cite{Z16} and references in there.

Zagier \cite{Zagier} conjectured that the dimension of  $\mathcal{Z}_w$, the  $\mathbb{Q}$-vector space spanned by the multizetas of weight $w$ is $d_w$, where $d_w = d_{w-2}+ d_{w-3}$ with $d_0 = 1$, $d_1 = 0$ and $d_2 = 1$. Later, Hoffman \cite{Hoffman1997} conjectured $\mathcal{Z}_w$ is generated by the basis of multizetas of weight $w$ of the form $\zeta(s_1, \dotsc, s_r)$ with $s_i \in \{2,3\}$. Brown \cite{Brown2012} proved that $\mathcal{Z}_w$ is generated by multizetas of weight $w$ of the form $\zeta(s_1, \dotsc, s_r)$ with $n_i \in \{2,3\}$ and so the dimension of $\mathcal{Z}_w$ is bounded above by $d_w$. 

\section{Multiple zeta values in positive characteristic}

We will look at the function field analog focusing in the multizeta values for $\F_q[t]$, where $\F_q$ is a finite field of characteristic $p$ consisting of $q$ elements. For a survey of work on  function field analogs of multizeta, with connections to Drinfeld modules and Anderson's $t$-motives  (see \cite{A86, G96, T} for background), we refer to the survey \cite{T17}. We first fix the notation and give the basic definitions.

\begin{tabular}{l l}
$\Z$  & \{integers\},\\
$\Z_+$ &  \{positive integers\},\\
$q$ &  a power of a prime $p$, \\
$\F_q$ & a finite field of $q$ elements,\\
$A$ & the polynomial ring $\F_q[t]$, $t$ a  variable\\
$A_+$ &   monics in  $A$,\\
$A_{d+}$ & $\{\text{elements of } A_+ \text{ of degree }d\}$, \\
$A_{<d}$ & $ \{\text{elements of } A \text{ of degree less than }d \}$, \\
$K$ &   the function field $\F_q(t)$,\\
$K_\infty$ & $\F_q((1/t)) = $ the completion of $K$ at $\infty$,\\
$[n]$ &   $ t^{q^n}-t$,\\
$D_n$ & $ \prod _{i= 0}^{n-1} (t^{q^n}-t^{q^i}) = [n][n-1]^q \dotsm
[1]^{q^{n-1} }$,\\
$\ell_n$ & $ \prod _{i= 1} ^n (t-t^{q^i}) = (-1)^n L_n = (-1)^n [n][n-1]\dotsm [1]$, \\
`even' & $  \text{multiple of } q-1$,\\
\end{tabular}

In \cite{C35} Carlitz introduced the  zeta values $\zeta(s)$ for $s$ a positive integer given by 
\begin{align*}
 \zeta(s) = \sum _{a\in A_+} \frac{1}{a^s} \in K_\infty.
\end{align*}
These are the Carlitz zeta values. See \cite{T} and references therein for more in this and basic analogies between function field and number field situations. Next, we recall definitions of power sums, iterated power sums, zeta and multizeta values \cite{T, Tm}.

For $k, k_i\in \Z_+$ and $d\geq 0$, consider the power sums (this is  $S_d(-k)$ in the notation of \cite{T})
\begin{align*}
S_d(k) := \sum _{ a\in A_{d+}} \frac{1}{a^{k}} \in K,
\end{align*}
and extend inductively to the iterated power sums 
\begin{align*}
S_d(k_1, \dotsc, k_r)  =  S_d(k_1) \, S_{<d}(k_2, \dotsc, k_r)
 =     S_d(k_1)  \sum_{d>d_2>\cdots>d_r}
S_{d_2}(k_2)\cdots S_{d_r}(k_r),& 
\end{align*}
where $S_{<d}=\sum_{i=0}^{d-1}S_i$ as the notation suggests.
For positive integers $s_i$, we consider the multizeta values
\begin{align*}
\zeta(s_1, \cdots, s_r):= \sum_{d=0}^{\infty} S_d(s_1, \cdots, s_r)=\sum  \frac{1}{a_1^{s_1}\cdots a_r^{s_r}}\in K_{\infty}, 
\end{align*}
where the second sum is over all $a_i\in A_+$ of degree $d_i$ such that
$d_1>\dotsb >d_r\ge 0$. We say that this multizeta value (or rather the tuple $(s_1, \dotsc, s_r)$) has depth $r$ and weight  $\sum s_i$. In depth one, we recover the Carlitz zeta.


Following \cite{LT14}, we call the $r$-tuple $(s_1,\dotsc, s_r)$ {\it zetalike}, if $\zeta(s_1, \dotsc, s_r)/\zeta(s_1+ \dotsb + s_r)$ is rational. A multizeta value of weight $w$ is called {\it Eulerian}, if  it is a rational multiple of $\tilde{\pi}^w$.

Since $\zeta(ps_1, \dotsc, ps_r)/\zeta(ps_1', \dotsc, ps_{r'}') = (\zeta(s_1, \dotsc, s_r)/\zeta(s_1', \dotsc, s_{r'}'))^p$, in all the discussion we can restrict to tuples of tuples where not all entries are divisible by $p$. 

\begin{remarks}[Trivial cases]
\
\begin{enumerate}
 \item 

If $\mathbf{s}$ is a zetalike $r$-tuple and $\mathbf{s'}$ is a zetalike $r'$-tuple and both are of the same weight, then $\zeta(\mathbf{s})/\zeta(\mathbf{s'})$ is rational.

\item If $\zeta(\mathbf{s})/\zeta(\mathbf{s''})$ and $\zeta(\mathbf{s'})/\zeta(\mathbf{s''})$ are rational, then $\zeta(\mathbf{s})/\zeta(\mathbf{s'})$ is rational. 
 \end{enumerate}

\end{remarks}

The analogues of Zagier-Hoffman's conjectures in positive characteristic were formulated by Thakur in \cite{T17} and by  Todd in \cite{Todd2018} as follows: If $d(w)$ denotes the dimension of the $K$-span of multizeta values of weight $w$, then $d(w)=2^{w-1}$ for $1 \le w <q$, $d(w) = 2^{w-1}-1$ for $w = q$ and for $w>q$, Todd conjectured $d(w) = \sum _{i=1}^q d(w-i)$. Thakur's conjecture says that the basis of this span 
is given by multizeta values  $\zeta(s_1, ..., s_r)$ of weight $w$, with 
$s_i\leq q$ for $1\leq i<r$ and $s_r<q$. Both these conjectures are proved recently by  Dac in  \cite{TuanNgoDac2021}.

In \cite{LT14} we proved (see also \cite{CPY19,  C17, T17}) some  multizeta families  to be Eulerian for the rational function field case, and conjectured that these are the only Eulerian multizetas, but could only prove and conjecture several zetalike families of  weights which are not $q$-even, without getting full characterization, even conjecturally.  

In \cite{LT2021}, we prove some relations between the multizeta values for positive genus function fields of class number one and conjecture some relations between them; also, we looked for possible rational ratios of multizeta of depth two or three of the same weight, not explained by our conjecture on the zeta-like family, but we did not find any.

Earlier, we looked at two-term linear relations with rational coefficients, where one term was zeta and one multizeta, now we generalize to both multizeta values, using the same continued fraction method to detect and use power series manipulations to prove.

In this paper, we  looked  for rational ratios of multizeta of depth two of the same weight, not explained by our theorems or conjectures in \cite{LT14}.  We did find several examples where $\zeta(a,b)/ \zeta(c,d)$ is rational with $a+b = c+d$, where neither the numerator nor the denominator are zetalike. We did find a relation between a multizeta of arbitrary depth $r \ge 2$ and a multizeta of depth two (see \cite{Lara2022Preprint})

\section{Results on multizeta  values}

We now state some relations between multizeta values of depth two. Proofs will given in the fourth section. 

\begin{theorem}
\label{thm}
For $A = \F_q[t]$, where $q$ is any prime power, we have

\begin{align}
\label{thm1}
 \zeta (q^{n+k}-q^n,\, q^{n+k}-1) = \frac{ [k]^{q^n}  }{[n+k]}\zeta(q^{n+k}-1, \, q^{n+k}-q^n),
\end{align}
where $n \ge 0$ and $k \ge 1$.

Let $k \ge 1$,  $1 \le s <q$, $0 \le k_i, l_i<k$. We have
\begin{align}
\label{thm2}
\zeta(q^k - \sum _{i=1}^s q^{k_i}, \, sq^k - \sum _{i=1}^s q^{l_i}   ) 
= \prod_{i=1}^s \frac{[k-k_i]^{q^{k_i}} }{ [k-l_i]^{q^{l_i}}  } 
\zeta(q^k - \sum _{i=1}^s q^{l_i},\, sq^k - \sum _{i=1}^s q^{k_i} )
\end{align}

Let $k\ge 1$,  $1 \le s_1\le s_2 \le q$, $s_1 \ne q$ and let $0 \le k_i<k$ and $0\le l_j<k$. We have
\begin{multline}
\label{thm3}
\zeta(q^k - \sum _{i=1}^{s_1} q^{k_i}, s_2q^k - \sum _{j=1}^{s_2} q^{l_j}   ) \\
= \prod_{i=1}^{s_1} \frac{[k-k_i]^{q^{k_i}} }{ [k-l_i]^{q^{l_i}}  } 
\zeta(q^k - \sum _{i=1}^{s_1} q^{l_i}, s_2q^k - \sum _{j=1}^{s_1} q^{k_j} -\sum_{j=s_1+1}^{s_2} q^{l_j }  ).
\end{multline}

\end{theorem}

\begin{remarks}
\
\begin{enumerate}

\item 
 In \eqref{thm2}, when $q = 2$, $k = 3$,  $s = 1$, $k_1 = 2$ and $l_1 = 1$,  we get the pair of tuples $(4,6), (6,4)$; note that we can not recover these tuples from \eqref{thm1}.

 \item 
 In \eqref{thm3}, when $q = 3$, $k =2$ and $s_1 = 1$, $s_2 = 2$, $k_1 = 1$ and $l_1 = 0$, $l_2 = 1$ we get the pair of tuples $(6,14), (8, 12)$; we can not recover these tuples from \eqref{thm2}.

\end{enumerate}
 
\end{remarks}

Next, the main theorem which implies the parts of  Theorem~\ref{thm}, but without giving the explicit rational factors there.

\begin{theorem}[Main theorem]
\label{main-thm}
Let $q$ be an arbitrary prime power. 
Let $1 \le s_1<q$, $1 \le s_2\le q$ and let $0 \le k_i< k$, $0\le l_j<k$. 
Let $a = q^k - \sum _{i = 1}^{s_1} q^{k_i} $ and $b = s_2 q^k - \sum _{j=1}^{s_2} q^{l_j}$. 
Choose $s_1$ elements from the list $k_1, \dotsc, k_{s_1}, l_1,  \dotsc, l_{s_2}$, say $k_{1}', \dotsc, k_{s_1}'$ and let $l_1',\dotsc, l_{s_2}'$ the remaining elements. Let $a' = q^k- \sum _{i=1}^{s_1} q^{k_i'} $ and $b' = s_2 q^k -\sum _{j=1}^{s_2} q^{l_j'}$. 
Then $\zeta(a,b)/\zeta(a',b')$ is rational. 

\end{theorem}

\begin{remark}
Let $a = q^k- \sum _{i=1}^{s_1} q^{k_i}$ where $1 \le s_1<q$, $0 \le k_i<k$ and $s_2 = q$ and $l_j = k-1$, $b = s_2 q^k - \sum _{j=1}^{s_2}q^{l_j} = q^{k+1}- q^k$ so that $\zeta(a,b)$ is zetalike by \cite[Theorem 3.1 $(1)$, p. 790]{LT14}. 
 
\end{remark}

\begin{remarks}
\
\begin{enumerate}

\item  When $q = 3$, the pair of tuples $(3,8), (5,6)$ is obtained from Theorem~\ref{main-thm} but not from \eqref{thm3}.


\item When $q = 4$, numbers like $5$,  $19$ or $20$ among others can not be obtained from \eqref{Sd}; also, numbers like $18$ or $21$ are not covered by \eqref{Sltd1};  tuples like ((4, 21), (7, 18)) or ((19, 45), (25, 39)) among many others are not covered by Theorem~\ref{main-thm} since its proof relies  precisely on those  formulas. See Table~\ref{tableq4}.

\end{enumerate}
 
\end{remarks}

Next theorem covers some  tuples not covered by the main theorem.

\begin{theorem}
\label{thm-q-power-of-2}

 Let $q$ be any power of $2$. Let $s \le q$ and let $k_1 \dotsc, k_{s-1} \in \{0,1\}$ with $k_1 \le \dotsb \le k_{s-1}$ (just to avoid redundancies). Let $b = sq^2 - \sum _{i=1}^{s-1} q^{k_i}-1$ and $b' = sq^2- \sum _{i=1}^{s-1}q^{k_i} - q$. Then 
 
\begin{align}
\label{f1}
\zeta(2, b)  &= \frac{[1]^q}{[2]} \zeta(q+1,b').
\end{align}
Also, 
\begin{align}
\label{f21}
 \zeta(2, 2q^2 - 3q + 1) = \frac{[1]^q}{[2]} \zeta(q + 1, 2q^2 - 4q + 2).
\end{align}

Let $s<q$ and let $k_i, l_i \in \{0,1\}$ such that $l_1 \le \dotsb \le l_s$, $k_1 \le \dotsb \le k_s$, $l_i \le k_i$ and $l_i<k_i$ for exactly one value of $i$. Then
\begin{align}
\label{f2}
 \zeta(q^2 - \sum _{i=1}^s q^{k_i}, 2q^2-3q+1) = \frac{[1]^q}{[2]} \zeta(q^2- \sum _{i=1}^s q^{l_i}, 2q^2-4q+2).
\end{align}

Let $0 \le i \le 2$ and $0 \le j \le 2-i$.
Then for $b =q^3 - iq^2 - jq-(q-i-j)$ and $b' = q^3 - iq^2 - (q-(2-j))q - (2-i-j)$ we have
\begin{align}
\label{f3}
\zeta(4q-2,b) = \frac{[1]^{(q-2)q}}{[2]^{q-2}}\zeta(q^2+q, b').
\end{align}

For $q>2$ a power of $2$ and $0 \le i \le 2$ we have
\begin{align}
\label{f4}
 \zeta(3q-1, q^3-(q-1)-q^i) &= \frac{[1]^{(q-1)q}}{[2]^{q-1} } \zeta(q^2+q, q^3-(q-1)q-q^i).
\end{align}
 
\end{theorem}

\begin{remarks}
\
\begin{enumerate}
 \item 
By specializing \eqref{f1}  to  $q = 2$,  we get the tuples $((2, 3), (3, 2))$, $((2, 5), (3, 4))$, $((2, 6), (3, 5))$ and
$((2,3), (3,2))$; these tuple are  obtained too from Theorem~\ref{thm} \eqref{thm3}.

\item 
By specializing \eqref{f21}  or \eqref{f2}  to $q = 2$ we get $((2,3), (3,2))$; from Theorem~\ref{thm} \eqref{thm3}, by taking $k = 2$, $s_1 = 1$, $s_2 = 2$, $k_1 = 1$, $l_1 = 0$ and $l_2 = 1$, we get the same tuple with ratio  $[2-1]^{2^1}/[2-0]^{2^{0}} = [1]^2/[2]$.

 \item 
By specializing \eqref{f1}  to  $q = 4$,  we get the tuples $((2, 15), (5, 12))$, $((2, 27), (5, 24))$ and many others;
None of these tuples are obtained from Theorem~\ref{main-thm} or Theorem~\ref{thm}.

\item By specializing \eqref{f21} to $q = 4$ we get $((2, 21), (5, 18))$ which is not covered by the main theorem.

\item By specializing \eqref{f2} to $q = 4$ we get $((4, 21), (7, 18))$,
 $((7, 21), (10, 18))$ and  others
None of these tuples are obtained from Theorem~\ref{main-thm} or Theorem~\ref{thm}.

\item By specializing \eqref{f3} to $q = 4$ we get 
$((14, 60), (20, 54))$,
$((14, 57), (20, 51))$ and others.
None of these tuples are covered directly by the main theorem.
The tuples $((7, 30), (10, 27))$, $((7, 27), (10, 24))$ and $((7, 15), (10, 12))$ are produced by \eqref{main-thm} when $q = 4$. 
From \eqref{f2}, we obtain $((7, 21), (10, 18))$; however, the tuples 
$((14, 57), (20, 51))$, $((14, 45), (20, 39))$ are not covered directly or indirectly by another theorem. 
 
\item 
By specializing \eqref{f4}  to $q = 4$ we get the tuples $((11,45), (20,36))$, $((11, 57), (20, 48))$ and $((11, 60), (20, 51))$ which are not covered by Theorem~\ref{main-thm}.

 \end{enumerate}

\end{remarks}

We finish the section with a conjecture. 

\begin{conjecture}
\label{conj-power-of-2}
Let $q$ be  a power of 2.   
Let $0 \le i \le 2$ and $0 \le j \le 2-i$.
Let $b =q^3 - iq^2 - jq-(q-i-j)$ and $b' = q^3 - iq^2 - (q-(2-j))q - (2-i-j)$.  Then for  $a \in \{4q, q^{2} + q - 1, q^{2} + q, q^{2} + 2q - 2, q^{2} + 2q - 1, q^{2} + 2q\}$,  we have 
\begin{align*}
\zeta(a,b) = \frac{[1]^{(q-2)q}}{[2]^{q-2}}\zeta(a+b-b', b') 
\end{align*}
\end{conjecture}

\section{Proof}

The following formulas, which are consequences of Theorems 1 and 3 in  \cite{LT2015},
will be used in the proofs of the theorems.

For $1 \le s < q$ and $0 \le k_i <k$, 
we have
\begin{align}
\label{Sd}
S_d(q^k- \sum_{i=1}^s q^{k_i})& = \ell_d^{(s-1)q^k} S_d(q^k-q^{k_1}) \dotsm
S_d(q^k -q^{k_s}).
\end{align}

For $1\leq s \leq q$, and any $0\leq k_i\leq k$,  we have
\begin{align}
\label{Sltd}
S_{<d} ( \sum _{i=1}^s(q^k-q^{k_i}) )
=
\prod _{i=1}^s S_{<d}(q^k -q^{k_i}).
\end{align}
As a consequence of \eqref{Sltd} we have the formula 
\begin{align}
\label{Sltd1}
 \ell _d^{m(q-1)} S_{<d}(m(q-1)) = [d]^{mq}/[1]^m, \quad m \le q. 
\end{align}
(See \cite[3.3.1, p. 2329]{Tm}).

We also recall Carlitz' evaluations (see e.g., \cite[3.3.1, 3.3.2]{Tm})

\begin{align}
\label{Sdqjminus1}
S_d(q^j-1) = {\ell_{d+j-1}}/{\ell_{j-1} \ell_d^{q^j}}
\end{align} 

\begin{align}
\label{Sltdqjminus1}
S_{<d}(q^j-1) = { \ell_{d+j-1}  }/{ \ell_j \ell_{d-1}^{q^j}},
\end{align}

Using  \eqref{Sdqjminus1} and \eqref{Sltdqjminus1}, by straight calculations we get
\begin{align}
\label{rel1}
 S_d(q^j-1) = \frac{[j] }{[d]^{q^j}}S_{<d}(q^j-1).
\end{align}

\begin{proof}[Proof of Theorem \ref{thm} \eqref{thm1}]

Let $a = (q^k-1)q^n$ and $b = q^{n+k}-1$. Since $S_d(s_1,s_2) = S_d(s_1)S_{<d}(s_2)$, using  \eqref{rel1}, we get
 $S_d(a,b) =  \frac{ [k]^{q^n}  }{[n+k]} S_d(b,a)$:
\begin{align*}
\frac{S_d(q^k-1)^{q^n} S_{<d}(q^{n+k}-1)} {S_{d}(q^{n+k}-1) S_{<d}(q^k-1)^{q^n}  }
& = \left( \frac{S_d(q^k-1) }{S_{<d}(q^k-1) } \right)^{q^n} \frac{ S_{<d}(q^{n+k}-1)  } { S_{d}(q^{n+k}-1) }\\
& = \frac{ [k]^ {q^n} }{ [d]^{q^{n+k}} } \frac{ [d]^{q^{n+k}} } { [n+k] }\\
& = \frac{[k]^ {q^n} }{ [n+k]}.
\end{align*}

By definition, we have 
\begin{align*}
 \zeta(a,b) = \sum _{d=1}^\infty S_d(a,b) = 
 \frac{ [k]^{q^n}  }{[n+k]} \sum _{d=1}^\infty S_d(b,a) =  \frac{ [k]^{q^n}  }{[n+k]} \zeta(b,a).
\end{align*}

{\it Proof of Theorem \ref{thm} \eqref{thm2}}. 
Let $a = q^k - \sum _{i=1}^s q^{k_i}$, $b = sq^k - \sum _{i=1}^s q^{l_i}$, $a' = q^k - \sum _{i=1}^s q^{l_i}$ and $b' = sq^k -  \sum _{i=1}^s q^{k_i}$. Since $S_d(a,b) = S_d(a)S_{<d}(b)$, by using \eqref{rel1}, we have 

\begin{align*}
\frac{ S_d(a)S_{<d}(b) }{S_d(a')S_{<d}(b') } & = 
\frac{\ell_d^{(s-1)q^k} \prod _{i=1}^s S_d(q^k- q^{k_i}) \prod _{i=1}^s S_{<d}(q^k -q^{l_i})    }{ \ell_d^{(s-1)q^k} \prod _{i=1}^s S_d(q^k - q^{l_i})  \prod _{i=1}^s S_{<d}(q^k - q^{k_i})  }\\
& = \prod _{i=1}^s   \frac{S_d(q^k -q^{k_i})  }{ S_{<d}(q^k - q^{k_i}) } \prod _{i=1}^s \frac{ S_{<d}(q^k -q^{l_i}) }{ S_d(q^k-q^{l_i}) }\\
& = \prod _{i=1}^s  \frac{ [k-k_i]^{q^{k_i}} }{[d]^{q^k}} \prod _{i=1}^s  \frac{[d]^{q^k}  }{  [k-l_i]^{q^{l_i}}  }\\
& = \prod_{i=1}^s \frac{[k-k_i]^{q^{k_i}} }{ [k-l_i]^{q^{l_i}}  }.
\end{align*}
By summing over $d$ the result follows.

{\it Proof of Theorem~\ref{thm} \eqref{thm3}}.
Let
\begin{align*}
a & =  q^k - \sum _{i=1}^{s_1} q^{k_i}, &
b & =  s_2q^k - \sum _{j=1}^{s_2} q^{l_j},\\
a' &= q^k - \sum _{i=1}^{s_1} q^{l_i}, &
b' &= s_2q^k -  \sum _{j=1}^{s_1} q^{k_j} -\sum _{j=s_1+1}^{s_2}q^{l_j}. 
\end{align*}
Note that $a+b = a'+b'$. We have
\begin{align*}
\frac{ S_d(a)S_{<d}(b) }{S_d(a')S_{<d}(b') } & = 
\frac{\ell_d^{(s_1-1)q^k} \prod _{i=1}^{s_1} S_d(q^k- q^{k_i}) \prod _{j=1}^{s_1} S_{<d}(q^k -q^{l_i}) \prod _{j=s_1+1}^{s_2} S_{<d}(q^k -q^{l_j})   }{ \ell_d^{(s_1-1)q^k} \prod _{i=1}^{s_1} S_d(q^k - q^{l_i})  \prod _{j=1}^{s_1} S_{<d}(q^k - q^{k_j}) \prod _{j=s_1+1}^{s_2} S_{<d}(q^k  -q^{l_j} )    }\\
& = \prod _{i=1}^{s_1}   \frac{S_d(q^k -q^{k_i})  }{ S_{<d}(q^k - q^{k_i}) } \prod _{i=1}^{s_1} \frac{ S_{<d}(q^k -q^{l_i}) }{ S_d(q^k-q^{l_i}) }\\
& = \prod _{i=1}^{s_1}  \frac{ [k-k_i]^{q^{k_i}} }{[d]^{q^k}} \prod _{i=1}^{s_1}  \frac{[d]^{q^k}  }{  [k-l_i]^{q^{l_i}}  }\\
& = \prod_{i=1}^{s_1} \frac{[k-k_i]^{q^{k_i}} }{ [k-l_i]^{q^{l_i}}  }.
\end{align*}
\end{proof}

\begin{proof}[Proof of Theorem~\ref{main-thm}]
We have that 
\begin{align*}
a+b = (s_2+1)q^k - \sum _{i=1}^{s_1}q^{k_i} -\sum _{j=1}^{s_2}q^{l_j} = (s_2+1)q^k - \sum _{i=1}^{s_1}q^{k_i'} -\sum _{j=1}^{s_2}q^{l_j'} = a'+b'. 
\end{align*}

If $k_i' \in \{k_1, \dotsc, k_{s_1}\}$ for all $i$, then $l_j'\in \{l_1, \dotsc, l_{s_2}\}$ for all $j$, and, therefore, $a = a'$ and $b = b'$, so that  $\zeta(a,b)/\zeta(a',b') = 1$. 

Assume there exists at least one $i$ such that $k_i'\notin \{k_1, \dotsc, k_{s_1}\}$ and let $a$ be the maximum number of indices $i$ such that $k_i' \in \{l_1, \dotsc, l_{s_2}\}$;
let  $v_1, \dotsc, v_a, v_1'\dotsc, v_a'$ be  such that $k_{v_i'}' = l_{v_i}$; therefore, $k_i' \in \{k_1, \dotsc, k_{s_1}\}$ if $i \notin \{v_1'\dotsc, v_a'\}$; on the other hand, for each $k_i' \in \{l_1, \dotsc, l_{s_2}\}$, there is a $l_i' \in \{k_1, \dotsc, k_{s_1}\}$;  it follows that there exist $w_1, \dotsc, w_a, w_1', \dotsc, w_a'$ such that $l_{w_i'}' = k_{w_i}$. 
Since $S_d(q^k-q^{k_i})/S_d(q^k-q^{k_j'})=1$ if $k_i = k_j'$ and $S_{<d} (q^k -q^{l_i})/S_{<}(q^k-q^{l_j'})=1$ if $l_i = l_j'$, we have
\begin{align*}
\frac{S_d(a,b)}{S_d(a',b')} & = 
\frac{\ell_d^{(s_1-1)q^k} \prod_{i=1}^{s_1} S_d(q^k - q^{k_i}) \prod_{j=1} ^{s_2} S_{<d}(q^k - q^{l_j})      }{ \ell_d^{(s_1-1)q^k} \prod_{i=1}^{s_1} S_d(q^k - q^{k_i'}) \prod_{j=1} ^{s_2} S_{<d}(q^k - q^{l_j'})       }\\
& = 
\prod _{i=1}^{a} \frac{S_d(q^k -q^{k_{w_i} })  }{S_{<d}(q^k -q^{l_{w_i'}'}  )  }
\prod _{i=1}^{a} \frac{S_{<d} (q^k - q^{l_{v_i} })  }{S_d(q^k-q^{k_{v_i'}'} )} \\
& = \prod _{i=1}^a
\frac{ [k-k_{w_i}]^ {q^{k_{w_i}} }  }{[d]^{q^k}}
\prod _{i=1}^a 
\frac{ [d]^{q^k} }{ [k - l_{v_i}]^{q^{l_{v_i}} } }\\
& = 
\prod _{i=1}^a 
\frac{ [k-k_{w_i}]^ {q^{k_{w_i}} }  } { [k - l_{v_i}]^{q^{l_{v_i}} } }.
\end{align*} 

\end{proof}

\begin{proof}[Proof of Theorem~\ref{thm-q-power-of-2}]
 

From \cite[3.3.1, p. 2329]{Tm}, we know that 
\begin{align*}
\ell_d^{q+b} S_d(q+b) = 1 - b\frac{[d]^q}{[1]}, \text{ if } 1\le b <q, 
\end{align*}
For $q$ a power of 2 and $b = 1$ or $q-1$ we have
\begin{align}
\label{Sd-q-plus-b}
\ell_d^{q+b} S_d(q+b) & = 1- \frac{[d]^q}{[1]}
=
\frac{[d+1]}{[1]}.
\end{align}

We use the generating function (see \cite[3.2]{Tm})
\begin{align*}
\frac{x}{\ell _d \binom{x}{q^d}} & = 1 + \sum _{k>0,\text{`even'}} S_{<d}(k)x^k
\end{align*}
for $S_{<d}(k)$ for $k$   `even' to get the formula
\begin{align}
\ell_d ^{(q+m)(q-1)}S_{<d}((q+m)(q-1)) &= -\frac{m[d]^{q^2+(m-1)q}[d-1]^{q^2}}{[1]^{q+m-1} [2]} + \frac{[d]^{(q+m)q} }{[1]^{q+m}},%
\label{mt1}
\end{align}
where $1\le   m \le q+1$  (Also, see \cite[p. 40]{Lara2009}). 

Since $2 q^{2} - 3 q + 1 =  (2q-1)(q-1)$, by specializing \eqref{mt1} to $m = q-1$ we have 
\begin{align}
\label{mt1-1}
\ell_d^{2 q^{2} - 3 q + 1} S_{<d}(2 q^{2} - 3 q + 1) & = 
\frac{[d]^{2q^2-2q}[d-1]^{q^2}}{[1]^{2q-2} [2]} + \frac{[d]^{2q^2-q} }{[1]^{2q-1}}\\
& = \frac{ [d+1] [d]^{2q^2-2q}  }{[1]^{q-1} [2]}. \nonumber
\end{align}


From \eqref{Sdqjminus1}  and \eqref{Sltdqjminus1} we have
\begin{align*}
\ell_d^{q^2-1} S_d(q^2-1) & = \frac{[d+1]}{[1]},\\ 
\ell_d^{q^2-1} S_{<d}(q^2-1) &= \frac{[d+1][d]^{q^2}}{[1][2] },\\ 
\ell_d^{q-1} S_{<d}(q-1) & = \frac{[d]^q}{[1]}.
\end{align*}

From \eqref{Sltd1} with $m =q-1$, we have that 
\begin{align*}
\ell_d^{2q^2-4q+2} S_{<d}(2q^2-4q+2) & = \ell_d^{2m(q-1)} S_{<d}(2m(q-1)) \\ 
&=\ell_d^{2m(q-1)} S_{<d}(q-1)^{2(q-1)} \\
& = \frac{[d]^{2(q-1)q} }{[1]^{2(q-1)}}.
\end{align*}

Since $2q^2-q-1 = (q^2-q) + (q^2-1)$, from \eqref{Sltd} we have 
\begin{align*}
\ell_d^{2q^2-q-1}S_{<d}(2q^2-q-1) & =\ell_d^{q^2-q} S_{<d}(q^2-q) \cdot \ell_d^{q^2-1} S_{<d}(q^2-1)\\
& = \frac{[d]^{q^2}}{[1]^q} \cdot \frac{[d+1][d]^{q^2}}{[1][2] }\\
& = \frac{[d+1][d]^{2q^2}}{[1]^{q+1}[2]}.
\end{align*}

{\it Proof of \eqref{f1}.} 
By \eqref{Sltd} we have
\begin{align*}
 S_{<d}(sq^2- \sum _{i=1}^{s-1}q^{k_i}-1) & = S_{<d}(q^2-1) \prod _{i=1}^{s-1} S_{<d}(q^2-q^{k_i}),\\
S_{<d}(sq^2- \sum _{i=1}^{s-1}q^{k_i}-q) & = S_{<d}(q^2-q)\prod _{i=1}^{s-1} S_{<d}(q^2-q^{k_i}).
\end{align*}
Therefore, 
\begin{align*}
 \frac{S_d(2,b)}{ S_d(q+1, b')} & = 
 \frac{S_d(2) S_{<d}(q^2-1) }{S_d(q+1)  S_{<d}(q^2-q) }\\
 & = 
\frac{\ell_d^2 S_d(2) \cdot \ell_d^{q^2-1} S_{<d}(q^2-1) }{\ell_d^{q+1} S_d(q+1) \cdot \ell_d^{q^2-q}  S_{<d}(q^2-q) }\\
& = \frac{
\frac {[d+1][d]^{q^2}}{[1][2]}   
}{
\frac{[d+1]}{[1]} \cdot \frac {[d]^{q^2}}{[1]^q} 
} \\
& = \frac{[1]^q}{[2]}.
\end{align*}

{\it Proof of \eqref{f21}.}
Let $a = 2$, $b = 2q^2-3q+1$, $a' = q+1$ and $b' = 2(q-1)^2$. 
Then, 
\begin{align*}
\frac{S_d(2,b)}{S_d(q+1, b')} & = 
\frac{\ell_d^2 S_d(2) \cdot \ell_d^{2q^2-3q+1} S_{<d}(2q^2-3q+1)  }
{ 
\ell_d^{q+1} S_d(q+1) \cdot \ell_d^{2q^2-4q+2}S_{<d}(2q^2-4q+2)
}\\
& = \frac{
\frac{ [d+1] [d]^{2q^2-2q}  }{[1]^{q-1} [2]}
}{
\frac{[d+1]}{[1]} \cdot \frac{ [d]^{2(q-1)q} }{[1]^{2(q-1)}}
}\\
& = \frac{[1]^{q}}{[2]}.
\end{align*}

{\it Proof of \eqref{f2}.}
Let $a = q^2-\sum_{i=1}^s q^{k_i}$, $b = 2q^2-3q+1$, $a' = q^2-\sum _{i=1}^s q^{l_i}$ and $b' = 2(q-1)^2$. 
Since
\begin{align*}
 \sum q^{k_i} - \sum q^{l_i} = \sum q^{l_i} (q^{k_i-l_i}-1) = q-1,
\end{align*}
it follows that $a+b = a'+b'$.
Let $j$ such that $0 = l_j <k_j=1$. Then
\begin{align*}
\frac{ S_d(q^2-\sum q^{k_i})}
{S_d(q^2-\sum q^{l_i}) } & = 
\frac{\ell_d^{(s-1)q^2} S_d(q^2-q)  \prod _{i \ne j} S_d(q^2 -q^{k_i})  }
{\ell_d^{(s-1)q^2} S_d(q^2-1) \prod _{i \ne j} S_d(q^2 -q^{l_i})   }
 = \frac{S_d(q^2-q)  }{ S_d(q^2-1)}.
\end{align*}

Therefore,
\begin{align*}
 \frac{S_d(a,b)}{S_d(a',b')}
 & = \frac{\ell_d^{q^2-q}  S_d(q^2-q) \cdot \ell_d^{2q^2-3q+1}  S_{<d}(2q^2-3q+1)  }{\ell_d^{q^2-1}  S_d(q^2-1) \cdot \ell_d^{2q^2-4q+2} S_{<d}(2q^2-4q+2)  }\\
 & = \frac{ 
\frac{ [d+1][d]^{2q^2-2q}} {[1]^{q-1}[2]}
 }
{ 
\frac{[d+1]}{[1]} \cdot \frac {[d]^{2q^2-2q}} {[1]^{2q-2}} 
} \\
& = \frac{[1]^q}{[2]}.
\end{align*}

{\it Proof of \eqref{f3}.}
Let $0 \le i \le 2$ and $0 \le j \le 2-i$. 
Write $ a= 4q-2$, $a' = q^2+q$, $b =q^3 - iq^2 - jq-(q-i-j)$ and $b' = q^3 - iq^2 - (q-(2-j))q - (2-i-j)$. If $q = 2$, then $a = a' = 6$ and $b = 6-3i-j = b'$ so that $\zeta(a,b) = \zeta(a',b')$ and $[1]^{(q-2)q}/[2]^{q-2} =1$ and there is nothing to prove. 
So, we assume $q>2$.

We have
\begin{align*}
\frac{S_{<d}(b)}{ S_{<d}(b')} & = \frac{S_{<d}(q^3 - iq^2 - jq-(q-i-j))}{S_{<d}(q^3 - iq^2 - (q-(2-j))q - (2-i-j)) }\\
& = \frac{ S_{<d}(q^2-q^2)^i  S_{<d}(q^2-q)^j S_{<d}(q^2-1)^{q-i-j}  }{S_{<d}(q^2-q^2)^i  S_{<d}(q^2-q)^{q-2+j}  S_{<d}(q^2-1)^{2-i-j} }\\
& = \frac{S_{<d}(q^2-1)^{q-2}}{S_{<d}(q-1)^{(q-2)q}}
\end{align*}
Let $a = 4q-2$ and $a' = q^2+q$. Then,
 \begin{align*}
\frac{S_d(a,b)  }{S_d(a',b')} 
& = \frac{ 
\ell_d^{2(2q-1)} S_d(2q-1)^2 \cdot \ell_d^{(q-2)(q^2-1)}  S_{<d}(q^2-1)^{q-2}
}{
\ell_d^{(q+1)q} S_d(q+1)^q  \cdot \ell_d^{(q-1)(q-2)q}  S_{<d}(q-1)^{(q-2)q} 
}\\
& = 
\frac{
\frac{[d+1]^2  }{[1]^2} \left( \frac{[d+1][d]^{q^2}}{[1][2]}  \right)^{q-2}
}{
\frac{[d+1]^q  }{[1]^q}  \left( \frac{ [d]^{q}} {[1]} \right)^{(q-2)q} 
}\\
& =
\frac{[1]^{(q-2)q}}{[2]^{q-2}}.
 \end{align*}

{\it Proof of \eqref{f4}.}
\begin{align*}
\frac{S_d(3q-1, q \cdot q^2 - (q-1)-q^i)}
{S_d(q^2+q, q\cdot q^2- (q-1)q-q^i) } & =  
\frac{
S_d(3q-1)  \prod_{i=1}^{q-1} S_{<d}(q^2-1) S_{<d}(q^2-q^i) 
}{
S_d(q+1)^q \prod_{i=1}^{q-1} S_{<d}(q^2-q) S_{<d}(q^2-q^i)
}\\
& = 
\frac{
S_d(3q-1) S_{<d}(q^2-1)^{q-1}
}{
S_d(q+1)^q S_{<d}(q-1)^{(q-1)q}
}\\
& = 
\frac{
\ell_d^{3q-1} S_d(3q-1) \cdot \ell_d^{(q^2-1)(q-1)}  S_{<d}(q^2-1)^{q-1}
}{
\ell_d^{q^2+q} S_d(q+1)^q \cdot \ell_d^{(q-1)(q-1)q}   S_{<d}(q-1)^{(q-1)q}
}\\
& = 
\frac{ \frac{[d+1] }{[1]} \left(\frac{[d+1][d]^{q^2} }{[1][2]} \right)^{q-1}      }
{ \frac{[d+1]^q}{[1]^q } \frac{[d]^{(q-1)q^2}  }{[1]^{(q-1)q} }   }\\
& = \frac{ [1]^{(q-1)q}}  {[2]^{q-1}}.
\end{align*}

\end{proof}

\section{Data}

  The numerical exploration to find two term linear relations with rational coefficients, where both terms are multizeta, was done following the method of \cite{LT14} using SageMath on a
  computer server,  using   the continued fractions  in $\F_q((1/t))$.

We only list primitive tuples of tuples. The tuples marked with * are covered by theorems or conjectures in \cite{LT14}; the tuples marked with + are covered by the theorems.

\subsection{Data for \texorpdfstring{$q=2$}{q2}}

\begin{center}
\begin{longtable}{ *{3}{l}} 
\toprule
\multicolumn{3}{ c}{$q = 2$. Depth two by depth and weight at most 100 }\\ \midrule
\endhead

\bottomrule
\endfoot

((1, 3), (2, 2))*    &  
((2, 3), (3, 2))+    &  
((2, 5), (3, 4))*, + \\  
((2, 6), (3, 5))*, + &  
((3, 5), (4, 4))*    &  
((4, 7), (7, 4))+    \\  
((6, 7), (7, 6))+     &  
((4, 11), (7, 8))*, + &  
((4, 13), (6, 11))+  \\  
((4, 13), (7, 10))+  &  
((6, 11), (7, 10))+  &  
((4, 14), (7, 11))+ \\ 
((6, 13), (7, 12))+ &  
((6, 14), (7, 13))+ &  
((8, 15), (15, 8))+ \\  
((12, 15), (15, 12))+ &  
((14, 15), (15, 14))+ &  
((8, 23), (15, 16))*, + \\  
((8, 27), (12, 23))+ &  
((8, 27), (15, 20))+ &
((12, 23), (15, 20))+ \\  
((8, 29), (14, 23))+ &  
((8, 29), (15, 22))+ &  
((14, 23), (15, 22))+ \\  
((8, 30), (15, 23))+ &  
((12, 27), (15, 24))+ &  
((12, 29), (14, 27))+ \\  
((12, 29), (15, 26))+ &  
((14, 27), (15, 26))+ &  
((12, 30), (15, 27))+ \\  
((14, 29), (15, 28))+ &  
((14, 30), (15, 29))+ &  
((16, 31), (31, 16))+ \\  
((24, 31), (31, 24))+ &  
((28, 31), (31, 28))+ &  
((30, 31), (31, 30))+ \\  
((16, 47), (31, 32))*, + &   

((16, 55), (24, 47))+  &     
((16, 55), (31, 40))+  \\     
((24, 47), (31, 40))+  &  
((16, 59), (28, 47))+  &     
((16, 59), (31, 44))+  \\     
((28, 47), (31, 44))+  &     
((16, 61), (30, 47))+   &   
((16, 61), (31, 46))+  \\     
((30, 47), (31, 46))+  &     
((16, 62), (31, 47))+  &     
((24, 55), (31, 48))+ \\   
((24, 59), (28, 55))+  &     
((24, 59), (31, 52))+  &     
((28, 55), (31, 52))+  \\     
((24, 61), (30, 55))+  &   
((24, 61), (31, 54))+  &     
((30, 55), (31, 54))+  \\     
((24, 62), (31, 55))+  &     
((28, 59), (31, 56))+  &   
((28, 61), (30, 59))+  \\     
((28, 61), (31, 58))+  &     
((30, 59), (31, 58))+  &     
((28, 62), (31, 59))+  \\   
((30, 61), (31, 60))+  &     
((30, 62), (31, 61))+  &     
((32, 63), (63, 32))+  \\     

\end{longtable}
\end{center}

\begin{center}
\begin{longtable} { *{2}l } 
\toprule
\multicolumn{2}{ c}{$q = 2$. Depth three by depth two and weight at most 64. }\\ \midrule
\endhead
\bottomrule
\endfoot

((1, 1, 2), (1, 3))* &
((1, 1, 2), (2, 2))* \\
((1, 2, 4), (2, 5))* &
((1, 2, 4), (3, 4))* \\
((1, 2, 5), (2, 6))* &
((1, 2, 5), (3, 5))* \\
((1, 2, 5), (4, 4))* &
((1, 3, 4), (2, 6))* \\
((1, 3, 4), (3, 5))* &
((1, 3, 4), (4, 4))* \\
((2, 2, 4), (3, 5))* &
((1, 4, 8), (3, 10))+ \\
((3, 4, 8), (4, 11))* &
((3, 4, 8), (7, 8))* \\
((3, 8, 16), (7, 20))+ &
((7, 8, 16), (8, 23))* \\
((7, 8, 16), (15, 16))* &
((7, 16, 32), (15, 40))+ \\
((15, 16, 32), (16, 47))* &
((15, 16, 32), (31, 32))* \\ 

\end{longtable}
\end{center}

\begin{center}
\begin{longtable}{ *{3} c  }\toprule
\multicolumn{3}{c }{
$q = 2$. Depth four by depth two and weight at most 32
} \\ \toprule
\endhead
 
((1, 1, 2, 4), (2, 6))* &  
((1, 1, 2, 4), (3, 5))* &
((1, 1, 2, 4), (4, 4))* \\
((1, 2, 4, 8), (4, 11))* &
((1, 2, 4, 8), (7, 8))*  &
((1, 3, 4, 8), (4, 12))* \\
((1, 3, 4, 8), (6, 10))*  &
((1, 3, 4, 8), (8, 8))*   &
((1, 4, 8, 16), (7, 22))+ \\
((3, 4, 8, 16), (8, 23))*  &
((3, 4, 8, 16), (15, 16))* &

 \\
\bottomrule
\end{longtable}
\end{center}

\begin{center}
\begin{longtable} { *{3}l } \toprule
\multicolumn{3}{ c}{$q = 2$. Depth four by depth three and weight at most 31. }\\ \midrule
\endhead

((1, 1, 2, 4), (1, 2, 5))*  &  
((1, 1, 2, 4), (1, 3, 4))*  &  
((1, 1, 2, 4), (2, 2, 4))*  \\
((1, 2, 4, 8), (3, 4, 8))*  &
((1, 3, 4, 8), (2, 4, 10))*  &  
((1, 3, 4, 8), (2, 6, 8))*  \\ 
((1, 3, 4, 8), (4, 4, 8))*   &
((3, 4, 8, 16), (7, 8, 16))* &
\\
\bottomrule
\end{longtable}
\end{center}

\subsection{Data for \texorpdfstring{$q=3$}{q3}}

\begin{center}
\begin{longtable} { *{3}l } \toprule  
\multicolumn{3}{ c}{$q = 3$. Depth two by depth two and weight at most 93} \\ \toprule
\endhead
\midrule
\endfoot

\bottomrule
\endlastfoot

((1, 8), (3, 6))*  &  
((3, 8), (5, 6))+  &  
((5, 8), (7, 6))+  \\  
((6, 8), (8, 6))+   &  
((3, 14), (5, 12))*, +  &  
((3, 16), (5, 14))+  \\ 
((3, 16), (7, 12))+  &  
((5, 14), (7, 12))+  &  
((6, 14), (8, 12))+  \\  
((5, 16), (7, 14))+  &  
((6, 16), (8, 14))+  &  
((3, 20), (5, 18))*, +  \\  
((3, 22), (5, 20))*, +  &  
((3, 22), (7, 18))*, +  &  
((5, 20), (7, 18))*, +  \\  
((6, 20), (8, 18))*, +  & 
((3, 24), (5, 22))*, +  &  
((3, 24), (7, 20))*, +  \\  
((5, 22), (7, 20))*, +  &
((5, 22), (9, 18))*    &  
((7, 20), (9, 18))*  \\  
((6, 22), (8, 20))+  &
((5, 24), (7, 22))+  &  
((6, 24), (8, 22))+  \\  
((9, 26), (17, 18))+  &  
((15, 26), (17, 24))+  &  
((15, 26), (23, 18))+  \\  
((17, 24), (23, 18))+  &  
((17, 26), (25, 18))+  &  
((18, 26), (26, 18))+  \\  
((21, 26), (23, 24))+  &  
((23, 26), (25, 24))+  &  
((24, 26), (26, 24))+  \\  
((9, 44), (17, 36))*, +  &  
((9, 50), (15, 44))+  &
((9, 50), (17, 42))+  \\  
((9, 50), (23, 36))+  &  
((15, 44), (17, 42))+  &  
((15, 44), (23, 36))+  \\  
((17, 42), (23, 36))+  &  
((9, 52), (17, 44))+  &  
((9, 52), (25, 36))+  \\ 
((17, 44), (25, 36))+  &  
((18, 44), (26, 36))+  &  
((15, 50), (17, 48))+  \\  
((15, 50), (21, 44))+  &  
((15, 50), (23, 42))+  &  
((17, 48), (21, 44))+  \\  
((17, 48), (23, 42))+  &  
((21, 44), (23, 42))+  &  
((15, 52), (17, 50))+  \\  
((15, 52), (23, 44))+  & 
((15, 52), (25, 42))+  &  
((17, 50), (23, 44))+  \\  
((17, 50), (25, 42))+  &  
((23, 44), (25, 42))+  &  
((18, 50), (24, 44))+  \\  
((18, 50), (26, 42))+  &  
((24, 44), (26, 42))+  &  
((17, 52), (25, 44))+  \\  
((18, 52), (26, 44))+  &  
((9, 62), (17, 54))*,+  &  
((21, 50), (23, 48))+  \\  
((21, 52), (23, 50))+  &  
((21, 52), (25, 48))+  &  
((23, 50), (25, 48))+  \\  
((24, 50), (26, 48))+  &  
((23, 52), (25, 50))+  &  
((24, 52), (26, 50))+  \\  
((9, 68), (15, 62))*,+  &  
((9, 68), (17, 60))*,+  & 
((9, 68), (23, 54))*,+  \\  
((15, 62), (17, 60))*,+ &  
((15, 62), (23, 54))*,+  &  
((17, 60), (23, 54))*,+  \\  
((9, 70), (17, 62))*,+  & 
((9, 70), (25, 54))*,+  &  
((17, 62), (25, 54))*,+  \\  
((18, 62), (26, 54))*,+  &
((9, 74), (15, 68))+  &
((9, 74), (17, 66))+  \\
((9, 74), (21, 62))+  &
((9, 74), (23, 60))+  &
((15, 68), (17, 66))+  \\
((15, 68), (21, 62))+  &
((15, 68), (23, 60))+  &
((17, 66), (21, 62))+  \\
((17, 66), (23, 60))+  &
((21, 62), (23, 60))+  &
((9, 76), (15, 70))+   \\
((9, 76), (17, 68))+  &
((9, 76), (23, 62))+  &
((9, 76), (25, 60))+  \\
((15, 70), (17, 68))+  &
((15, 70), (23, 62))+  &
((15, 70), (25, 60))+  \\
((17, 68), (23, 62))+  &
((17, 68), (25, 60))+  &
((23, 62), (25, 60))+  \\
((18, 68), (24, 62))+  &
((18, 68), (26, 60))+  &
((24, 62), (26, 60))+  \\
((9, 78), (17, 70))+  &
((9, 78), (25, 62))+  &
((17, 70), (25, 62))+  \\
((18, 70), (26, 62))+  &

((15, 74), (17, 72)) + &
((15, 74), (21, 68)) + \\
((15, 74), (23, 66)) + &
((17, 72), (21, 68)) + &
((17, 72), (23, 66)) + \\
((21, 68), (23, 66)) + &

((15, 76), (17, 74))+ &
((15, 76), (21, 70))+ \\
((15, 76), (23, 68))+ &
((15, 76), (25, 66))+ &
((17, 74), (21, 70))+ \\
((17, 74), (23, 68))+ &
((17, 74), (25, 66))+ &
((21, 70), (23, 68))+ \\
((21, 70), (25, 66))+ &
((23, 68), (25, 66))+ &

((18, 74), (24, 68))+ \\
((18, 74), (26, 66))+ &
((24, 68), (26, 66))+ &

((15, 78), (17, 76))+ \\
((15, 78), (23, 70))+ &
((15, 78), (25, 68))+ &
((17, 76), (23, 70))+ \\
((17, 76), (25, 68))+ &
((23, 70), (25, 68))+ &


\end{longtable}
\end{center}

\begin{remark}
There are no zetalike tuples  $(a,b)$ with $82 \le a+b \le 93$ (See \cite[4.4.1]{KL2016}). 
\end{remark}

\begin{center}
\begin{longtable} { *{2}c } \toprule  
\multicolumn{2}{ c}{$q = 3$. Depth three by depth two and weight at most 20 and others} \\ \midrule
\endhead
\bottomrule
\endfoot

((1, 2, 6), (1, 8))* &
((1, 2, 6), (3, 6))* \\
((2, 18, 54), (8, 66))+ & 
((8, 54, 162), (26, 198))+ \\
((728, 4374, 13122), (2186, 16038))+

 \end{longtable}
\end{center}


\subsection{Data for \texorpdfstring{$q=4$}{q4}}

\begin{center}
\begin{longtable}{ *{4}{l}}
\caption{
The tuples marked with + are covered by the main theorem. The tuples marked with \dag{}   are covered by \eqref{f1}. Tuples marked with \ddag{} are covered by \eqref{f4}. The tuple ((2, 21), (5, 18)) is covered by \eqref{f21}. Tuples marked with C are covered by Conjecture~\ref{conj-power-of-2}.
 }\label{tableq4}\\
\toprule
\multicolumn{4}{ c}{$q = 4$. Depth two by depth two and weight at most 104}\\ \midrule
\endfirsthead
\midrule
\multicolumn{4}{ c}{$q = 4$. Depth two by depth two and weight at most 104}\\ \midrule
\endhead

 \midrule
\endfoot

\bottomrule
\endlastfoot

((1, 15), (4,12))*    & 
((2, 15), (5, 12))\dag   & 
((4, 15), (7, 12))+   \\ 
((7, 15), (10, 12))+  &
((2, 21), (5, 18))*   & 
((8, 15), (11, 12))+  \\
((4, 21), (7, 18))\ddag   & 
((10, 15), (13, 12))+ &
((11, 15), (14, 12))+  \\
((12, 15), (15, 12))+ &
((7, 21), (10, 18))\ddag  &
((2, 27), (5, 24))*,\dag   \\
((8, 21), (11, 18))\ddag  &
((4, 27), (7, 24))*,+   &
((10, 21), (13, 18))\ddag \\
((11, 21), (14, 18))\ddag &
((12, 21), (15, 18))\ddag &
((2, 30), (5,27))*, \dag   \\
  
((4, 30), (7, 27))+   &
((7, 27), (10, 24))+  &
((8, 27), (11, 24))+  \\
((7, 30), (10, 27))+  &
((7, 30), (13, 24))+  &
((10, 27), (13, 24))+ \\
((8, 30), (11, 27))+  &                    
((11, 27), (14, 24))+  &                    
((12, 27), (15, 24))+  \\
((10, 30), (13, 27))+  &
((2, 39), (5, 36))\dag  &
((11, 30), (14, 27))+   \\ 
((12, 30), (15, 27))+   &
((4, 39), (7, 36))*,+     &
((2, 42), (5, 39))\dag     \\
((4, 42), (7, 39))*,+    &
((7, 39), (10, 36))*,+  &

((2, 45), (5, 42))\dag    \\
((8, 39), (11, 36))*,+   &
((4, 45), (7, 42))+    &
((4, 45), (10, 39))+   \\
((4, 45), (13, 36))+   &
((7, 42), (10, 39))+   &
((7, 42), (13, 36))+   \\
((10, 39), (13, 36))+  & 
((19, 30), (25, 24))C  &
((8, 42), (11, 39))+   \\
((11, 39), (14, 36))+  &

((12, 39), (15, 36))+ &
((7, 45), (10, 42))+  \\
((7, 45), (13, 39))+   &
((10, 42), (13, 39))+  &
((2, 51), (5, 48))\dag    \\
((8, 45), (11, 42))+   &
((8, 45), (14, 39))+   &
((11,42), (14, 39))+   \\
((23, 30), (29, 24))C  & 
((12, 42), (15, 39))+  &

((4, 51), (7, 48))*,+  \\
((10, 45), (13, 42))+ &
((2, 54), (5, 51))\dag    &
((11, 45), (14, 42))+  \\
((11, 45), (20, 36))  &
((12, 45), (15, 42))+ &
((4, 54), (7, 51))*,+\\ 
((7, 51), (10, 48))*,+ &

((2, 57), (5, 54))\dag   &
((8, 51), (11, 48))*,+  \\
((14, 45), (20, 39))  &

((4, 57), (7, 54))*,+  &
((4, 57), (10, 51))*,+ \\
((4, 57), (13, 48))*,+  &
((7, 54), (10, 51))*,+  &
((7, 54), (13, 48))*,+  \\
((10, 51), (13, 48))*,+ &
((16, 45), (22, 39))C   &
((19, 42), (25, 36))C  \\

((2, 60), (5, 57))\dag    &
((8, 54), (11, 51))*,+  &
((11, 51), (14, 48))*,+ \\
((12, 51), (15, 48))*,+ &

((4, 60), (7, 57))*,+  &
((4, 60), (13, 51))*,+ \\
((7, 57), (10, 54))*,+ &
((7, 57), (13, 51))*,+ &
((10, 54), (13, 51))*,+  \\
((19, 45), (25, 39))C  &

((8, 57), (11, 54))+  &
((8, 57), (14, 51))+  \\
((11, 54), (14, 51))+  &
((20, 45), (26, 39))C  &
((23, 42), (29, 36))C  \\

((12, 54), (15, 51))+  &
((7, 60), (10, 57))+  &
((7, 60), (13, 54))+  \\
((10, 57), (13, 54))+  &
((22, 45), (28, 39))C  &
((8, 60), (11, 57))+  \\
((11, 57), (14, 54))+  &
((11, 57), (20, 48))  &
((23, 45), (29, 39))C \\
((12, 57), (15, 54))+  &
((24, 45), (30, 39))C  &
((10, 60), (13, 57))+  \\

((8, 63), (23, 48))  &
((11, 60), (14, 57))+  &
((11, 60), (20, 51))  \\
((14, 57), (20, 51))  &

((12, 60), (15, 57))+  &
((16, 57), (22, 51))C \\
((19, 54), (25, 48))C  &
((11, 63), (23, 51))  &
((19, 57), (25, 51))C \\

((14, 63), (23, 54)) &
((14, 63), (29, 48)) &
((20, 57), (26, 51))C \\
((23, 54), (29, 48))C &
((16, 63), (31, 48))+ &
((19, 60), (25, 54))C \\
((22, 57), (28, 51))C &
((23, 57), (29, 51))C &
((24, 57), (30, 51))C \\
((20, 63), (23, 60))  &
((20, 63), (29, 54))  &
((23, 60), (29, 54))C \\
((7, 78), (13, 72))  &
((22, 63), (31, 54)) &
((26, 63), (29, 60)) \\
((7, 84), (13, 78))  &
((28, 63), (31, 60))+ &
((28, 63), (43, 48))+ \\
((31, 60), (43, 48))+ &

((31, 63), (46, 48))+ &
((8, 87), (23, 72))  \\
((32, 63), (47, 48))+ &
((4, 93), (10, 87))  &
((7, 90), (13, 84))  \\
((19, 78), (25, 72)) &
((7, 93), (13, 87))  &

((8, 93), (14, 87)) \\
((8, 93), (23, 78)) &
((8, 93), (29, 72)) &
((14, 87), (23, 78)) \\
((14, 87), (29, 72)) &
((23, 78), (29, 72)) &
((38, 63), (53, 48)) \\

((16, 87), (31, 72)) &
((19, 84), (25, 78)) &
((40, 63), (43, 60))+ \\
((40, 63), (55, 48))+ &
((43, 60), (55, 48)) &

\end{longtable}
\end{center}

Since 
\begin{align*}
R_1 =  \frac{\zeta(11,45)}{\zeta(20,36)} = 
  \frac{\zeta(11,45) / \zeta(7,21)^2}{\left(\zeta(10,18)/\zeta(7,21)\right)^2 }
\end{align*}
and both $\zeta(11,45)/\zeta(14,42)$ and $\zeta(10,18)/\zeta(7,21)$ rational, it follows that $R_1$ is rational. In a similar way, 
\begin{align*}
R_2 = \frac{\zeta(11,57)}{\zeta(20,48)} = \frac {\zeta(11,57)/\zeta(7,27)^2} { (\zeta(10,24)/\zeta(7, 27) )^2} 
\end{align*}
and it follows that $R_2$ is rational.  


\subsection{Data for \texorpdfstring{$q=5$}{q5}}

\begin{center}
\begin{longtable}{ *{3}{l}} \toprule
\multicolumn{3}{ c}{$q = 5$. Depth two by depth two and weight at most 80}\\ \midrule
\endhead

((1, 24), (5,20))*    & 
((5, 24), (9, 20))+   &
((9, 24), (13, 20))+  \\
((10, 24), (14, 20))+  &
((13, 24), (17, 20))+  &
((14, 24), (18, 20))+  \\
((15, 24), (19, 20))+   &
((17, 24), (21, 20))+   &
((18, 24), (22, 20))+   \\
((19, 24), (23, 20))+   &
((20, 24), (24, 20))+   &
((5, 44), (9, 40))*,+  \\
((5, 48), (9, 44))+  &
((5, 48), (13, 40))+  &
((9, 44), (13, 40))+  \\
((10, 44), (14, 40))+  &
((9, 48), (13, 44))+  &
((9, 48), (17, 40))+  \\
((13, 44), (17, 40))+  &
((10, 48), (14, 44))+  &
((10, 48), (18, 40))+  \\
((14, 44), (18, 40))+  &

((15, 44), (19, 40))+  &
((13, 48), (17, 44))+  \\
((13, 48), (21, 40))+  &
((17, 44), (21, 40))+  &
((14, 48), (18, 44))+  \\
((14, 48), (22, 40))+  &
((18, 44), (22, 40))+  &
((15, 48), (19, 44))+  \\
((15, 48), (23, 40))+  &
((19, 44), (23, 40))+  &
((20, 44), (24, 40))+  \\

((17, 48), (21, 44))+  &
((18, 48), (22, 44))+  &
((19, 48), (23, 44))+  \\
((20, 48), (24, 44))+  &
((5, 64), (9, 60))*,+  &

((5, 68), (9, 64))*,+  \\
((5, 68), (13, 60))*,+  &
((9, 64), (13, 60))*,+  &
((10, 64), (14, 60))*,+  \\
((5, 72), (9, 68))+  &
((5, 72), (13, 64))+  &
((5, 72), (17, 60))+  \\
((9, 68), (13, 64))+  &
((9, 68), (17, 60))+  &
((13, 64), (17, 60))+  \\
((10, 68), (14, 64))+  &
((10, 68), (18, 60))+  &
((14, 64), (18, 60))+  \\
((15, 64), (19, 60))+  &

\\ \bottomrule
\end{longtable}
\end{center}

\subsection{Data for \texorpdfstring{$q=9$}{q9}}

\begin{center}
\begin{longtable}{ *{3}{l}} \toprule
\multicolumn{3}{ c}{$q = 9$. Depth two by depth two and weight at most 34 and others}\\ \midrule
\endhead
((72, 80), (80, 72))+
\\ \bottomrule
\end{longtable}
\end{center}

\section{Acknowledgments}

I am grateful to thank Dinesh Thakur for their careful reading and for  suggesting numerous improvements to this paper.

\bibliographystyle{alpha}

\end{document}